\documentclass[12pt]{article}
\nonstopmode
\RequirePackage[colorlinks,citecolor=blue,urlcolor=blue,linkcolor=blue]{hyperref}
\hypersetup{
colorlinks = true,
citecolor=blue,
urlcolor=blue,
linkcolor=blue,
pdfauthor = {Daniel Hackmann, Alexey Kuznetsov},
pdfkeywords = {stable processes, supremum, Mellin transform, continued fractions},
pdftitle = {A note on the series representation for the density of the supremum of a stable process},
pdfpagemode = UseNone
}\usepackage{graphicx,xspace,colortbl}
\usepackage{amsmath,amsthm,amsfonts}
\usepackage{color}
\usepackage{fancybox}
\usepackage{epsfig}
\usepackage{subfig}
\usepackage{pdfsync}
    \def\qed{\hfill$\sqcap\kern-8.0pt\hbox{$\sqcup$}$\\}
    \def\beq{\begin{eqnarray}}
    \def\eeq{\end{eqnarray}}
    \def\beqq{\begin{eqnarray*}}
    \def\eeqq{\end{eqnarray*}}

    \def\q{{\mathbb Q}}

    \def\p{{\mathbb P}}
    \def\e{{\mathbb E}}
    \def\r{{\mathbb R}}
    \def\aa{\mathcal A}

    \def\i{{\textnormal i}}

	\newtheorem{theorem}{Theorem}
	\newtheorem{lemma}{Lemma}

\title{A note on the series representation for the density of the supremum of a stable process}
\author{D. Hackmann
\\ \\
Dept. of Mathematics and Statistics\\  York University \\
4700 Keele Street 
\\Toronto, ON \\ M3J 1P3,  Canada 
\and 
A. Kuznetsov
\thanks{{Research supported by the
Natural Sciences and Engineering Research Council of Canada.}}  \\ \\
Dept. of Mathematics and Statistics\\  York University \\
4700 Keele Street 
\\Toronto, ON \\ M3J 1P3,  Canada 
 }

\date{current version: April 19, 2013}

\begin{document}

\maketitle

\begin{abstract}
\bigskip
An absolutely convergent double series representation for the density of the supremum of $\alpha$-stable L\'evy process is given in \cite[Theorem 2]{HubKuz2011} for almost all irrational $\alpha$. This result cannot be made stronger in the following sense: the series does not converge absolutely when $\alpha$ belongs to a certain subset of irrational numbers of Lebesgue measure zero (see \cite[Theorem 2]{Kuz2013}). Our main result in this note shows that {\it for every} irrational $\alpha$ there is a way to rearrange the terms of the double series, so that it converges to the density of the supremum. We show how one can establish this stronger result by introducing a simple yet non-trivial modification in the original proof of \cite[Theorem 2]{HubKuz2011}.
\end{abstract}

{\vskip 0.5cm}
{\vskip 0.5cm}

\newpage

%****************************************************************************************************************
%****************************************************************************************************************
%****************************************************************************************************************

\section{Introduction}

%****************************************************************************************************************
%****************************************************************************************************************
%*************************************************************************************************************

Consider an $\alpha$-stable L\'evy process $X$, defined by the characteristic exponent
\beq\label{def_Psi}
\Psi(z)=-\ln \e \left[ e^{\i z X_1} \right] = |z|^{\alpha} \times \left( e^{\pi \i \alpha (\frac{1}{2}-\rho)}  {\bf 1}_{\{z>0\}}+e^{-\pi \i \alpha (\frac{1}{2}-\rho)} {\bf 1}_{\{z<0\}} \right), \;\;\; z\in \r. 
\eeq
Here $(\alpha,\rho)$  belong  to the set of admissible parameters
\beqq
\aa=\{\alpha \in (0,1), \; \rho \in (0,1)\} \cup  \{\alpha\in(1,2), \; \rho \in [1-\alpha^{-1}, \alpha^{-1}]\}.
\eeqq
It can be shown (see \cite{Kuz2013} and Section 4.6 in \cite{Zolotarev1986}) that with this parameterization we have $\rho=\p(X_1>0)$ 
and that any stable process with $\alpha\ne 1$ can be rescaled so that it's characteristic exponent is written in the form \eqref{def_Psi}.

Our main object of interest is the density 
$p(x)$ of the supremum $
S_1:=\sup\{X_u\; : 0\le u \le 1\}$. 
Let us summarize what is currently known about this function. 
\begin{itemize}
\item[(i)] The simplest case is when $X$ belongs to one of the Doney classes ${\mathcal C}_{k,l}$ (here $k,l \in {\mathbb Z}$), which are  
 defined by the relationship $\rho+k=l/\alpha$ (see \cite{Doney1987}). In this case an absolutely convergent series representation for $p(x)$ is given in \cite[Theorem 10]{Kuz2011}. From now on we assume that $X$ does not belong to one of the Doney classes. 
\item[(ii)] When $\alpha$ is rational, \cite[Theorem 3]{Kuz2013} gives an explicit formula for the Mellin transform of $p(x)$, 
which is expressed in terms of elementary functions and the dilogarithm function ${\textnormal {Li}}_2(z)$. Unfortunately, 
there seems to be little hope of obtaining an explicit series representation for $p(x)$: This problem is equivalent to evaluating the residues of  the Mellin transform, which in this case has multiple poles and computing these residues in closed form seems to be impossible. 
\item[(iii)]  When $\alpha$ is irrational, we  define sequences $\{a_{m,n}\}_{m\ge 0,n\ge 0}$ and  $\{b_{m,n}\}_{m\ge 0,n\ge 1}$ as
follows
\beq\label{def_a_mn}
a_{m,n}:=\frac{(-1)^{m+n} }{\Gamma\left(1-\rho-n-\frac{m}{\alpha}\right)\Gamma(\alpha\rho+m+\alpha n)}
\prod\limits_{j=1}^{m} \frac{\sin\left(\frac{\pi}{\alpha} \left( \alpha \rho+ j-1 \right)\right)} {\sin\left(\frac{\pi j}{\alpha} \right)} 
\prod\limits_{j=1}^{n} \frac{\sin(\pi \alpha (\rho+j-1))}{\sin(\pi \alpha j)},
\eeq
\beq\label{def_b_mn}
b_{m,n}:=\frac{\Gamma\left(1-\rho-n-\frac{m}{\alpha}\right)\Gamma(\alpha\rho+m+\alpha n) }{\Gamma\left(1+n+\frac{m}{\alpha}\right)\Gamma(-m-\alpha n)}
a_{m,n}.
\eeq
As was established in  \cite{HubKuz2011}, there exists a set of irrational numbers ${\mathcal L}$, which is  uncountable, dense, and has Lebesgue measure zero, such that for all irrational $\alpha \notin {\mathcal L}$
we have
\beq\label{eqn_p1}
p(x)& = &  x^{-1-\alpha } \sum\limits_{n\ge 0} \sum\limits_{m\ge 0}b_{m,n+1} x^{-m-\alpha n}, \;\; {\textnormal { if }} \alpha \in (0,1), \\
\label{eqn_p2}
p(x)& = &  x^{\alpha\rho-1} \sum\limits_{n\ge 0} \sum\limits_{m\ge 0} a_{m,n} x^{ m+\alpha n}, \;\;\;\;\;\;\;  {\textnormal { if }} \alpha \in (1,2).
\eeq
\item[(iv)]  Theorem 2 in \cite{Kuz2013} states that the previous result cannot be substantially 
improved: There exists an uncountable dense subset  $\tilde {\mathcal L} \subset {\mathcal L}$, such that for all 
$\alpha \in \tilde {\mathcal L}$ and almost all $\rho$ the series in \eqref{eqn_p1}, \eqref{eqn_p2} do not converge absolutely for
all $x>0$. 
\end{itemize}

This situation is clearly deficient, since we do not have a useful expression for $p(x)$ if $\alpha \in {\mathcal L}$ and 
we do not know whether the series \eqref{eqn_p1}, \eqref{eqn_p2} converge if $\alpha \in {\mathcal L}\setminus \tilde {\mathcal L}$. Also, determining whether $\alpha$ belongs to ${\mathcal L}$ or $\tilde{\mathcal L}$ 
is problematic, as this would require full knowledge of the continued fraction representation of $\alpha$
(see \cite[Proposition 1]{HubKuz2011} and \cite[Proposition 1]{Kuz2013}). Since it is impossible to have absolute convergence of the series for all irrational $\alpha$, 
our only remaining possibility is to try to find some sort of ``conditional" convergence. 
Absolute convergence implies that the order of summation does not matter. In the absence of absolute convergence, the way in which we compute the partial sums of the double series becomes very important. Our main result 
shows that the right way to compute the partial sums in  
\eqref{eqn_p1}, \eqref{eqn_p2} is over triangles $\{ m+\alpha n < C \; : \; m\ge0, n\ge 0\}
\subset {\mathbb Z}^2$, and that one can
always find an increasing sequence $C_k \to +\infty$ (which depends crucially on the arithmetic properties of $\alpha$), such that the partial sums will converge to $p(x)$.

%****************************************************************************************************************
%****************************************************************************************************************
%****************************************************************************************************************

\section{Main result}

%****************************************************************************************************************
%****************************************************************************************************************
%*************************************************************************************************************

For $x\in \r$, let $[x]$ denote the largest integer not greater than $x$ and let $\{x\}:=x-[x]$ denote the fractional part of $x$.  
The continued fraction representation (see \cite{Khinchin}) is defined as
\beqq	
 x=[a_0;a_1,a_2,\dots]=a_0+\cfrac{1}{a_1+\cfrac{1}{a_2+ \dots }}
 \eeqq
where $a_0 \in {\mathbb Z}$ and $a_i \in {\mathbb N}$ for $i\ge 1$. 
The coefficients of the continued fraction can be computed recursively as follows: Define $x_1:=\{x\}$ and $x_{i+1}:=\{1/x_i\}$, $i\ge 1$, then
$a_0=[x]$ and $a_i=[1/x_i]$, $i\ge 1$. For $x\notin \q$ the continued fraction has infinitely many terms; truncating it after $n$ steps results in  a rational number $p_n/q_n:=[a_0;a_1,a_2,...,a_n]$, which provides the best rational approximation  
to $x$ (among all rational numbers with denominators not greater than $q_n$) and is called the $n$-th convergent (see \cite[Theorem 17]{Khinchin}).  The numerators and denominators of convergents are known to satisfy the two-term recurrence relation
\beq\label{continued_fraction_recurrence}
\begin{cases}
p_n=a_n p_{n-1} + p_{n-2}, \;\;\; p_{-1}=1, \;\;\; p_{-2}=0, \\
q_n=a_n q_{n-1} + q_{n-2}, \;\;\; q_{-1}=0, \;\;\; q_{-2}=1.
\end{cases}
\eeq
The following theorem is our main result in this note. 
\begin{theorem}\label{thm_main}
Assume that $\alpha \notin \q$. Then for all $x>0$ 
\begin{equation}
p(x)=
\begin{cases} \label{eqn_p_0_infty}
 & \displaystyle x^{-1-\alpha } 
 \lim_{k\rightarrow\infty} \sum_{\substack {m +1+ \alpha (n+\frac{1}{2})  < q_k\\ m \geq 0,\,n \geq 0}} b_{m,n+1} x^{-m-\alpha n}, \;\; {\textnormal { if }} \alpha \in (0,1), \\ \\
 & \displaystyle  x^{\alpha\rho-1} \lim_{k\rightarrow\infty} \sum_{\substack{m +1+ \alpha (n+\frac{1}{2})  < q_k\\m \geq 0,\,n \geq 0}}  a_{m,n} x^{ m+\alpha n}, \;\;\;\;\;\;\;  {\textnormal { if }} \alpha \in (1,2),
\end{cases}
\end{equation}	
where $a_{m,n}$ and $b_{m,n}$ are defined by (\ref{def_a_mn}) and (\ref{def_b_mn})
and  $q_k=q_k(2/\alpha)$ is the denominator of the $k$-th convergent for $2/\alpha$.
\end{theorem}

To obtain the proof of Theorem \ref{thm_main} one would need to modify only one step in the proof of \cite[Theorem 2]{HubKuz2011}. Instead of reproducing the long and technical proof of the latter result in full entirety, we only explain here which step has to be modified.

Assume that $\alpha \in (1,2)$ and that $\alpha \notin \q$. Combining formulas (17), (21), (24) and (26)
in \cite{HubKuz2011} we find that for every integer $k\ge 1$
\beq\label{px_sum_1}
p(x)=
 x^{\alpha\rho-1} \sum_{\substack{m + \alpha (n+\frac{1}{2})  < k \\m \geq 0,\,n \geq 0}}  a_{m,n} x^{ m+\alpha n} + 
 {\textrm{e}}_k(x),
\eeq
where the error term $ {\textrm{e}}_k(x)$ can be bounded from above as follows
\beq\label{error_estimate}
| {\textrm{e}}_k(x)|< (A(1+x))^k \times e^{- \epsilon k \ln(k)} 
\times\prod\limits_{l=1}^k 
 \bigg |\sec\left(\frac{\pi l}{\alpha}  \right)\bigg | \;,
\eeq
for some constants $A>0$ and $\epsilon>0$ (which can depend on $(\alpha,\rho)$ but not on $x$ or $k$). 
 The main step in the proof \cite[Theorem 2]{HubKuz2011} 
 is to construct a set ${\mathcal L} \subset \r\setminus \q$ (which is dense,  
uncountable, of Lebesgue measure zero), such that 
for all $\alpha \notin {\mathcal L}$ we have an upper bound
\beq\label{estimate_product}
\prod\limits_{l=1}^k 
 \bigg |\sec\left(\frac{\pi l}{\alpha}  \right)\bigg |<B 3^k, \;\;\; k\ge 1,
\eeq
for some constant $B=B(\alpha)$. 
Combining \eqref{estimate_product} with \eqref{error_estimate} implies that $ {\textrm{e}}_k(x)\to 0$
as $k\to +\infty$, which gives us 
\beq\label{px_sum_2}
p(x)=
 x^{\alpha\rho-1} \lim\limits_{k\to +\infty} \sum_{\substack{m + \alpha (n+\frac{1}{2})  < k \\m \geq 0,\,n \geq 0}}  a_{m,n} x^{ m+\alpha n}.
\eeq
To obtain the series representation \eqref{eqn_p2} we use formula  \eqref{def_a_mn} and
Proposition 1 and Lemma 1 in \cite{HubKuz2011} and check that for $\alpha \notin {\mathcal L}$
the series \eqref{eqn_p2} converges absolutely, 
therefore the order of summation in \eqref{px_sum_2} does not matter, and the sum  can be rewritten in the form \eqref{eqn_p2}. 

 The assumption $\alpha \notin {\mathcal L}$ is crucial for deriving inequality
 \eqref{estimate_product} and formula \eqref{px_sum_2}. The upper bound \eqref{estimate_product} is not true 
 for all irrational $\alpha$, and for a suitable $\alpha$ the product in \eqref{estimate_product} can grow arbitrarily fast. The following example illustrates this phenomenon. 
 
\vspace{0.25cm}
\noindent
{\bf Example 1:}
We define $\tau$ via its continued fraction representation
\beqq
\tau=[a_0;a_1,a_2,\dots]=[1,2,2^4,2^{1089},\dots]
\eeqq
where the coefficients $a_n$ are defined as $a_{n+1}=2^{q_n^2}$, $n\ge -1$ and the numerators $p_n$ and the denominators $q_n$
 of the $n$-th convergent are computed recursively via \eqref{continued_fraction_recurrence}. 
We find the first few terms of $p_n$ and $q_n$ to be
\beqq
[p_0,p_1,p_2,\dots]=[1,3,49,\dots], \qquad \qquad [q_0,q_1,q_2,\dots]=[1,2,33,\dots].
\eeqq
Let us take $\alpha=2/\tau\approx 1.34693878...$.
Since $a_n$ are even integers for $n\ge 1$, it follows from 
\eqref{continued_fraction_recurrence}  that $p_n$ are all odd numbers, so that we can write $p_n=2 r_n+1$ for some integers $r_n$.  
From Theorem 13 in \cite{Khinchin} we know that 
\beqq
|q_n \tau - p_n|< \frac{1}{q_{n+1}}=\frac{1}{a_{n+1}q_n+q_{n-1}}<\frac{1}{a_{n+1}}=2^{-q_n^2}, 
\eeqq
therefore $|q_n/\alpha-r_n-1/2|<2^{-q_n^2}$. 
Using the inequality $|\cos(\pi x)|\le \pi |x-1/2|$, we conclude that 
\beqq
\bigg|\sec\left(\frac{\pi q_n}{\alpha}  \right)\bigg |=
\frac{1}{|\cos\left(\pi \left(\frac{q_n}{\alpha}-r_n \right)\right)|} \ge 
\frac{1}{\pi |\frac{q_n}{\alpha}-r_n-\frac{1}{2} |}>\frac{2^{q_n^2}}{\pi},
\eeqq
therefore for this particular choice of $\alpha$ the full sequence of products in \eqref{estimate_product} cannot be bounded from above by any exponential function of $k$. 
\vspace{0.25cm}

As the previous example demonstrates, one can construct $\alpha$ so that the full sequence of products in \eqref{estimate_product}
can grow arbitrarily fast. What is remarkable is that for any irrational $\alpha$ there is always 
an infinite subsequence of these products which are bounded from above by an exponential function of the index. 

 \begin{lemma}\label{cos} Assume that $\tau \notin \q$ and $\tau>0$. 
There exists a constant $C=C(\tau)>0$ such that
 for all $k\ge 1$ 
\beq\label{sec} 
\prod_{l=1}^{q_k - 1}\vert \sec (\pi l \tau) \vert &\leq C 6^{q_k}. 
\eeq
where $q_k=q_k(2\tau)$ is the denominator of the $k$-th convergent for $2\tau$. 
\end{lemma}

\begin{proof}
We use the following result (see \cite[Lemma 4]{Buslaev} or \cite[Lemma 4]{Petruska1992353}): 
for any $\beta>0$, $\beta\notin \q$, 
\beq \label{bus}
\lim_{k \rightarrow \infty}\frac{1}{q_k}\sum_{l=1}^{q_k - 1}\ln(2\vert\sin(\pi l \beta)\vert) = 0,
\eeq
where $q_k=q_k(\beta)$. 
This is equivalent to
\beqq
\lim_{k \rightarrow \infty} \left[ \prod_{l=1}^{q_k - 1}\frac{1}{\vert\sin(\pi l \beta)\vert} \right]
^{\frac{1}{q_k}} = 2,
\eeqq
which implies the existence of a constant $C=C(\beta)>0$, such that for all $k\ge 1$
\beqq
 \prod_{l=1}^{q_k - 1}\frac{1}{\vert\sin(\pi l x)\vert}  < C 3^{q_k}. 
\eeqq
Using the identity $\sin(\pi l \beta) = 2\sin(\pi l \tfrac{\beta}{2}) \cos(\pi l \tfrac{\beta}{2})$ we conclude
\beqq
 \prod_{l=1}^{q_k - 1}\frac{1}{\vert\cos(\pi l \tfrac{\beta}{2})\vert} < 
 C 3^{q_k} 2^{q_k-1} \prod_{l=1}^{q_k - 1} \vert\sin(\pi l \tfrac{\beta}{2})\vert<C6^{q_k}. 
\eeqq
Taking $\beta=2\tau$ we obtain \eqref{sec}.
\end{proof}

Combining \eqref{sec}  with \eqref{px_sum_1} and \eqref{error_estimate} gives us the statement of Theorem \ref{thm_main} for $\alpha \in (1,2)$. The proof for $\alpha\in(0,1)$ can be obtained in a similar  way.

%**************************************************************************************************
%**************************************************************************************************
%**************************************************************************************************

%**************************************************************************************************
%**************************************************************************************************
%**************************************************************************************************

\end{document}